\documentclass[12pt]{article}

\usepackage{amssymb,amsmath,amsbsy, amsthm}
\usepackage[english]{babel}
\usepackage{amscd}
\usepackage{pb-diagram}
\usepackage{pstricks,pst-node}
\usepackage{amsfonts}
\usepackage{latexsym}
\usepackage{pstricks,pst-node}
\usepackage{tikz}
\usepackage{fullpage}

\newtheorem{theorem}{Theorem}[section]
\newtheorem{lemma}[theorem]{Lemma}
\newtheorem{corollary}[theorem]{Corollary}
\newtheorem{claim}[theorem]{Claim}

\newtheorem{proposition}[theorem]{Proposition}
\newtheorem{example}[theorem]{Example}
\newtheorem{question}[theorem]{Question}
\theoremstyle{definition}
\newtheorem{definition}[theorem]{Definition}

\numberwithin{equation}{section}


\newcommand{\N}{\mathbb{N}}
\newcommand{\R}{\mathbb{R}}

\newcommand{\T}{\mathcal{T}}

\begin{document}
\date{}
\title{Continuity of the Jones' set function $\T$}         
\author{Javier Camargo and Carlos Uzc\'ategui}

\maketitle

\begin{abstract}
Given a continuum $X$, for each $A\subseteq X$, the Jones' set function $\T$ is defined by $\T(A)=\{x\in X : \text{for each subcontinuum }K\text{ such that }x\in \textrm{Int}(K), \text{ then }K\cap A\neq\emptyset\}.$ We show that $\mathcal{D}=\{\T(\{x\}):x\in X\}$ is decomposition of $X$ when $\T$ is continuous. We present a characterization of the continuity of $\T$ and  answer several open questions posed by D. Bellamy. 
\end{abstract}

\section{Introduction}

Given a continuum $X$, for each $A\subseteq X$, the Jones' set function $\T$ is defined by 
$$\T(A)=\{x\in X : \text{for each subcontinuum }K\text{ such that }x\in \textrm{Int}(K), \text{ then }K\cap A\neq\emptyset\}.
$$ 
The set function $\T$ was defined by  F. Burton Jones \cite{jones} in order to study some properties related with aposyndesis of continua. Since then it has been studied extensively as an aid to classify continua (see for instance \cite{bellamy1}, \cite{bellamy1.5}, \cite{bellamy2}, \cite{bellamy3}, \cite{jones}, \cite{Leo}, \cite{macias1}, \cite{macias1.5}, \cite{macias2006}, \cite{Macias0}  and \cite{Macias}).

A continuum is a compact connected and nonempty metric space. Given a continuum $X$, $2^{X}$ denotes the collection of all closed and nonempty subsets of $X$; $2^X$ is itself a continuum if it is topologized by the Hausdorff metric. It is known that $\T(A)\in 2^{X}$, for each $A\in 2^{X}$ and thus is natural to look for conditions under which   $\T\colon 2^{X}\to 2^{X}$ is continuous. This problem has been addressed in several places  (see  \cite{bellamy1}, \cite{bellamy1.5}, \cite{macias1.5}, \cite{macias2006}, \cite{macias1}, \cite{Macias0}) and is the objective of this article.   There are two trivial cases where $\T$ is continuous: (a) If $X$ is  locally connected, then $\T(A)=A$ for all $A\in 2^{X}$ and  (b) if  $X$  is indecomposable, then $\T(A)=X$ for all $A\in 2^{X}$. Bellamy \cite{bellamy1} asked if any nonlocally connected continuum for which $\T$ is continuous had to be indecomposable,  later he proved this is not the case by showing that $\T$ is continuous  for the circle of pseudo-arcs  \cite{bellamy1.5}.    Mac\'ias \cite{macias1.5} generalized Bellamy's example  showing that  if  $X$ is one dimensional, has a terminal decomposition in pseudo-arcs and its decomposition space is locally connected, then $\T$ is continuous.

From the very beginning it was recognized that $\mathcal{D}=\{\T(\{x\}):x\in X\}$ plays a crucial role in the study topological properties of continua using the function $\T$.
Bellamy \cite[Theorem 5, p. 9]{bellamy1.5} shows that if there exist a locally connected continuum $Y$ and a monotone and open map $f\colon X\to Y$ such that $\T(A)=f^{-1}(f(A))$, for each $A\in 2^{X}$, then $\T$ is continuous and $\mathcal{D}$ is clearly a decomposition of $X$. Later Mac\'ias \cite{Macias0} proved the converse, if  $\T$ is continuous and  $\mathcal{D}$ is a decomposition  of $X$, then the  conditions in the hypothesis of  Bellamy's theorem hold. In this paper, we prove that the continuity of $\T$ suffices, that is to say, it  implies that $\mathcal{D}$ is a decomposition (see Theorem \ref{teorema3}).  This was the missing piece to answer several questions posed  by Bellamy \cite[p. 390]{houston} and  by  Mac\'ias in  \cite[Chapter 7]{Macias} (see Corollary \ref {corolario1}, Theorem \ref{teorema3}, Theorem \ref{teorema descomponible}, Corollary \ref{problema7.25}).

Mac\'ias   \cite{macias1} proved that if $X$ is a continuously irreducible continuum, the set function $\T$ is continuous. We  will show that in the class of decomposable irreducible continua, the continuously irreducible continua are the only one for which  the set function $\T$ is continuous (see Theorem \ref{irreducible}).

\section{Preliminaries}

If $X$ is a metric space, then given $A\subseteq X$ and $\epsilon>0$, the open ball about $A$ of radius $\epsilon$ is denoted by $B(A,\epsilon),$ the closure of $A$ is denoted by $\overline{A}$ and its interior is denoted by $A^{\circ}$.  Given a sequence $(x_n)_{n\in\N}$ in $X$ and $x\in X$, $x_n\rightarrow x$ means that $(x_n)_{n\in\N}$ converges to $x$. In this paper every map will be a continuous function. If $f\colon X\to Y$ is a map and $A\subseteq X$, $f|_A$ denotes the restriction of $f$ to $A$.

Given a topological space $X$, a decomposition of $X$ is a family $\mathcal{G}$ of nonempty and mutually disjoint subsets of $X$ such that $\bigcup\mathcal{G}=X$. A decomposition $\mathcal{G}$ of a topological space $X$ is said to be \textit{continuous} if the quotient map $q\colon X\to X/\mathcal{G}$ is both closed and open.

A \textit{compactum} is a compact metric space. A \textit{continuum} is a compact connected and nonempty metric space. Given a continuum $X$, we define $$2^{X}=\{A\subseteq X: A\text{ is closed and nonempty}\}.$$ For each $A,B\in 2^{X}$, we define $$H(A,B)=\inf\{r>0:A\subseteq N(B,r)\text{ and }B\subseteq N(A,r)\},$$ where $N(D,s)=\{x\in X: d(x,z)<s,\text{ for some }z\in D\}$, $D\in 2^{X}$ and $s>0$. $H$ is the Hausdorff metric on $2^{X}$. It is known that $2^{X}$ is a continuum.

Let $C_1,...,C_n$ be subsets of $X$. We define 
$$
\langle C_1,...,C_n\rangle=\{A\in 2^X: A\subseteq \bigcup_{i=1}^{n}C_i\ \text{ and }A\cap C_i\neq\emptyset, \text{ for each }i\in\{1,...,n\}\}.
$$
Then  $\{\langle U_1,...,U_n\rangle:U_i\text{ is open of }X\text{ for each }i\}$ is a base of the topology generated by the Hausdorff metric in $2^{X}$. 

\begin{definition}
Let $X$ be a compactum. Define $\T\colon\mathcal{P}(X)\to \mathcal{P}(X)$ by
\begin{multline*}
\T(A)=X\setminus\{x\in X:\text{there exists a subcontinuum }W\text{ of }X\\ \text{ such that } x\in W^{\circ}\subseteq W\subseteq X\setminus A\},
\end{multline*}
for each $A\in\mathcal{P}(X).$ The function $\T$ is called \textit{Jones' set function $\T$}.
\end{definition}

It is not difficult to show that $\T(A)$ is closed, for each $A\in\mathcal{P}(X)$. Thus, the restriction $\T|_{2^{X}}\colon 2^X\to 2^{X}$ is well defined. When we say that $\T$ is continuous, we refer to $\T|_{2^X}$. It is known that $\T$ is upper semicontinuous \cite[Theorem 3.2.1]{Macias}. We say that $\T$ is \textit{idempotent} provided that $\T^{2}(A)=\T(\T(A))=\T(A)$, for each $A\in \mathcal{P}(X)$. We say that $X$ is \textit{idempotent on closed sets} provided that $\T^{2}(A)=\T(A)$, for each $A\in 2^{X}$. Clearly, if $\T$ is  idempotent, then it is idempotent on closed sets.

\begin{example}
There exists a continuum $X$ such that $X$ is idempotent on closed sets, but it is not idempotent.
\end{example}

We define $[a,b]=\{x\in\R^{2}:x=at+b(1-t),\ t\in [0,1]\}$, for each $a,b\in\R^{2}$. Let $X=\bigcup_{i=0}^{\infty}[v,u_i]$, where $v=(0,1)$, $u_0=(0,0)$ and $u_n=(\frac{1}{n},0)$, for each $n\in\N$. $X$ is known as the harmonic fan. 
Let $U=[v,u_1]\setminus \{v\}\subseteq X$. It is not difficult to see that $\T(U)=[v,u_1]$ and $\T^{2}(U)=[v,u_1]\cup [v,u_0]$. Hence, $\T^{2}(U)\neq \T(U)$ and $\T$ is not idempotent. Let $A\in 2^{X}$. It is easy to show that $$\T(A)=\begin{cases} A &\text{ if }A\cap [v,u_0]=\emptyset; \\ A\cup [w,u_0] &\text{ if } w\in A\cap [v,u_0] \text{ and }[v,w]\cap A=\{w\}.\end{cases}$$ Thus, we may verify that $\T^2(A)=\T(A)$, for each $A\in 2^X$; i. e., $\T$ is idempotent on closed sets.
\bigskip

A continuum $X$ is \textit{$\T-$additive} provided that for each pair, $A$ and $B$, of closed subsets of $X$, $\T(A\cup B)=\T(A)\cup\T(B)$. 
Furthermore, $X$ is \textit{$\T-$symmetric} if for each closed sets $A$ and $B$, we have that $A\cap\T(B)=\emptyset$ if and only if $\T(A)\cap B=\emptyset$. Finally, a continuum $X$ is \textit{point $\T-$symmetric} provided that for each $p,q\in X$, $p\notin \T(\{q\})$ if and only if $q\notin\T(\{p\})$.

\section{Main results}

In this section, we prove in Theorem \ref{teorema2} that if $\T$ is continuous, then $\mathcal{D}=\{\T(\{x\}):x\in X\}$ is a continuous decomposition of $X$ such that the quotient space $X/\mathcal{D}$ is locally connected. We show some interesting consequences of this result.

\begin{proposition}\label{semicontinuidad}
Let $X$ be a continuum and let $A_n, A, L\in 2^X$, for each $n\in\N$. If $A_n\rightarrow A$ and $\T(A_n)\rightarrow L$, then $L\subseteq \T(A)\subseteq \T(L)$.
\end{proposition}

\begin{proof}
We show first that $L\subseteq \T(A)$. Let $x\in L$ and $W$ be a continuum such that $x\in W^{\circ}$.  Since  $\T(A_n)\rightarrow L$, there is $n_0$ such that $\T(A_n)\cap W^{\circ}\neq\emptyset$ for all $n\geq n_0$. Then, by the defintion of $\T$,  $W\cap A_n\neq\emptyset$ for all $n\geq n_0$. By compactness, we can find an increasing sequence of integer  $(n_k)_{k\in\N}$ and $z_k\in W\cap A_{n_k}$ such that $z_k\rightarrow z$ for some $z\in X$. Thus $z\in A\cap W$. Since $W$ was arbitrary, then $x\in \T(A)$.  We have shown that $L\subseteq \T(A)$.

Now we show that $\T(A)\subseteq \T(L)$.  Since $A_n\subseteq \T(A_n)$, for each $n\in\N$, $A_n\rightarrow A$ and  $\T(A_n)\rightarrow L$, we have that $A\subseteq L$. Therefore, $\T(A)\subseteq \T(L)$.
\end{proof}

In \cite[Theorem 3.2.8]{Macias}, it is proved that if $\T$ is continuous then $\T$ is idempotent. The following result characterizes the continuity of $\T$ in terms of  idempotence.

\begin{theorem}\label{continuous}
Let $X$ be a continuum. Then, $\T$ is continuous if and only of $\T$ is idempotent on closed sets and $\T(2^X)$ is closed in $2^X$.
\end{theorem}

\begin{proof}
Notice that if $\T$ is continuous, then $\T$ is idempotent, by \cite[Theorem 3.2.8]{Macias}, and clearly $\T(2^X)$ is closed in $2^X$.
Conversely, suppose $\T$ is idempotent on closed sets and $\T(2^X)$ is closed in $2^X$. We will show that $\T$ is continuous.
Suppose $A_n, A\in 2^X$,  $A_n\rightarrow A$. By passing to a subsequence we assume that $\T(A_n)\rightarrow L$ for some $L\in 2^X$.  By Proposition \ref{semicontinuidad}, $L\subseteq \T(A)\subseteq \T(L)$. Since $\T(2^X)$ is closed in $2^X$, there is $B\in 2^X$ such that $L=\T(B)$. But $\T(L)=\T(\T(B))=\T(B)=L$ as $\T$ is idempotent on closed sets. Hence $L=\T(A)$.

\end{proof}

We will need the following result of Mac\'ias.

\begin{lemma}
\label{lema1}  (\cite[Theorem 3.7]{Macias0})
Let $X$ be a continuum such that $\T$ is idempotent on closed sets. If $x\in X$, then there exists $x_0\in X$ such that $\T(\{x_0\})\subset \T(\{x\})$ and  $\T(\{z\})=\T(\{x_0\})$ for each $z\in \T(\{x_0\})$.
\end{lemma}

\begin{theorem}\label{teorema1}
Let $X$ be a continuum. If $\T|_{F_1(X)}\colon F_1(X)\to 2^X$ is continuous and $\T$ is idempotent on closed sets, then $\mathcal{D}=\{\T(\{x\}):x\in X\}$ is a decomposition of $X$.
\end{theorem}

\begin{proof}
Let $${M}=\{x\in X:\text{ for each }z\in\T(\{x\}),\ \T(\{z\})=\T(\{x\})\}.$$
Observe that $\T(\{x\})\subseteq {M}$, for each $x\in {M}$. Also, by Lemma \ref{lema1}, ${M}\neq\emptyset$ and $\T(\{x\})\cap M\neq\emptyset$, for each $x\in X$.

\begin{claim}\label{claim1}
$M$ is closed in $X$.
\end{claim}

To show the claim, let $(x_n)_{n=1}^{\infty}\subseteq M$ and  $x\in X$  be such that  $x_n\rightarrow x$.  We will prove that $x\in M$. Let $z\in\T(\{x\})$.  Since $\T|_{F_1(X)}$ is continuous, $\T(\{x_n\})\rightarrow \T(\{x\})$. Hence, there exists $z_n\in\T(\{x_n\})$, for each  $n\in\N$, such that  $z_n\rightarrow z$. Thus, $\T(\{z_n\})\rightarrow \T(\{z\})$, by the continuity of  $\T|_{F_1(X)}$. Observe that  $\T(\{z_n\})=\T(\{x_n\})$, for each $n\in\N$, as $x_n\in M$.  Therefore, $\T(\{z\})=\T(\{x\})$ and  $x\in M$. This proves  Claim \ref{claim1}.

\begin{claim}\label{claim2}
$M=X$
\end{claim}

To prove the claim we assume that $M\neq X$ to get a contradiction.  We consider two cases:

(1)  Suppose $M$ is disconnected. Let $M=M_1\cup M_2$, where  $M_1$ and  $M_2$ are disjoint nonempty compact sets. Let $U_1$ and  $U_2$ be open sets such that  $M_1\subseteq U_1$, $M_2\subseteq U_2$ and  $U_1\cap U_2=\emptyset$. For each $x\in M_1$, as   $\T(\{x\})$ is connected  and  $x\in \T(\{x\})$, then  $\T(\{x\})\subseteq M_1\subseteq U_1$. Since $\T$ is upper semicontinuous, for each $x\in M_1$, there is  an open set $V_x\subseteq U_1$ such that  $x\in V_x$ and  $\T(\{z\})\subseteq U_1$, for each  $z\in V_x$. Let  $V_1=\bigcup\{V_x:x\in M_1\}$. It is clear that  $V_1\subseteq U_1$,  $M_1\subseteq V_1$ and $\T(\{x\})\subseteq U_1$  for each  $x\in V_1$. Analogously, we can find an open set $V_2\subseteq U_2$ such that  $M_2\subseteq V_2$ and  $\T(\{z\})\subseteq U_2$  for each $z\in V_2$.  Consider the following sets:
\[
\begin{array}{lcl}
R & = & \{x\in X:\T(\{x\})\cap M_1\neq\emptyset \mbox{ and } \T(\{x\})\cap M_2\neq\emptyset\}\\
    & \cong & \T^{-1}|_{F_1(X)}(\langle X,M_1\rangle\cap\langle X,M_2\rangle),
    \\\\
R_i & = & \{x\in X:\T(\{x\})\subseteq X\setminus M_i\}\cong \T^{-1}|_{F_1(X)}(\langle X\setminus M_i\rangle), \mbox{for $i\in\{1,2\}$}.
\end{array}
\]
Notice that $R$ is a closed set in $X$ and $R_i$ is an open set in $X$ for $i\in\{1,2\}$.

Also, $V_2\subseteq R_1$, $V_1\subseteq R_2$ and  $(V_1\cup V_2)\cap R=\emptyset$. Then $R_i\neq\emptyset$, $i\in\{1,2\}$. By Lemma \ref{lema1}, if $x\in X$, then  $\T(\{x\})\cap M\neq\emptyset$. Hence, $X\subseteq R\cup R_1\cup R_2$ and thus  $X=R\cup R_1\cup R_2$. Notice that if $x\in R_1\cap R_2$, then  $\T(\{x\})\subseteq X\setminus (M_1\cup M_2)$, which contradicts  Lemma \ref{lema1}. Thereofore, $R_1\cap R_2=\emptyset$. Since $X$ is connected and  $R_1$ y $R_2$ are open,  then $R\neq\emptyset$.

There is $x\in R$ such that $x\in \overline{R_1\cup R_2}$, without lost of generality, we assume that there is  $(x_n)_{n=1}^{\infty}\subseteq R_2$ such that  $x_n\rightarrow x$. Since $\T(\{x\})\cap M_2\neq\emptyset$  and $M_2\subseteq V_2$,  we have that $\T(\{x\})\in\langle X,V_2\rangle$. By the continuity of $\T$, there is  $k\in\N$  such that $\T(\{x_k\})\in \langle X,V_2\rangle$. Hence there is  $z\in \T(\{x_k\})\cap V_2$. By the choice  of $V_2$, $\T(\{z\})\subseteq U_2$.  By \cite[Theorem 3.7]{Macias0},  $\T(\{z\})\cap M\neq\emptyset$. As  $M_1\cap U_2=\emptyset$, then  $\T(\{z\})\cap M_2\neq\emptyset$.  Since $\T(\{z\})\subseteq \T(\T(\{x_k\}))=\T(\{x_k\})$, then $\T(\{x_k\})\cap M_2\neq\emptyset$ which contradicts  that  $x_k\in R_2$.

\bigskip

(2) Suppose $M$ is connected.  Let  $y\in X\setminus M$. By  Lemma \ref{lema1}, there is $x_0\in M$ such that  $\T(\{x_0\})\subseteq \T(\{y\})$. Let  $U$ and $V$ be open sets of  $X$  such that $M\subseteq U$, $y\in V$ and $U\cap V=\emptyset$. Since $\T$ is upper semicontinuous,  there is an open set $W$ of  $X$ such that  $x_0\in W$ and $\T(\{z\})\subseteq U$, for each  $z\in W$. Let $K=\overline{\bigcup\{\T(\{z\}):z\in W\}}\cup M$. Since $\T(\{z\})\cap M\neq\emptyset$, for each $z\in W$, and  $M$ is connected,  we have that  $\bigcup\{\T(\{z\}):z\in W\}\cup M$ is connected. Thus,  $K$ is a subcontinuum of $X$. It is clear that $K\subseteq \overline{U}\subseteq X\setminus V$ and  $x_0\in W\subseteq K$. Therefore, $K\cap \{y\}=\emptyset$  and $x_0\notin \T(\{y\})$, which contradicts that $\T(\{x_0\})\subseteq \T(\{y\})$.

\bigskip

This completes the proof that $X=M$, which clearly implies  that  $\{\T(\{x\}):x\in X\}$ is a decomposition of  $X$.
\end{proof}

\begin{theorem}\label{teorema2}
Let $X$ be a continuum. If $\T\colon 2^{X}\to 2^X$ is continuous, then $\mathcal{D}=\{\T(\{x\}):x\in X\}$ is a continuous decomposition such that $X/\mathcal{D}$ is locally connected.
\end{theorem}

\begin{proof}

Since $\T$ is continuous, $\T$ is idempotent, by \cite[Theorem 3.2.8]{Macias}. Thus, $\mathcal{D}$ is a decomposition, by Theorem \ref{teorema1}. It is clear that since $\T$ is continuous, $\T|_{F_1(X)}$ is continuous and $\mathcal{D}$ is a continuous decomposition. Finally, $X/\mathcal{D}$ is locally connected, by \cite[Theorem 3.4]{Macias0}.
\end{proof}

If $\mathcal{D}=\{\T(\{x\}):x\in X\}$ is a decomposition of $X$, then it is easy to see that $p\notin \T(\{q\})$ if and only if $q\notin \T(\{p\})$, for $p,q\in X$, $p\neq q$; i.e., $X$ is point $\T-$symmetric. Thus, by Theorem \ref{teorema2} and  \cite[Corollary 3.2.15]{Macias}, we have the following result which give us a positive answer to Question 7.2.1 of \cite{Macias}.

\begin{corollary}\label{corolario1}
Let $X$ be a continuum. If $\T\colon 2^{X}\to 2^{X}$ is continuous, then $X$ is both $\T-$symmetric and $\T-$additive.
\end{corollary}

\begin{theorem}\label{locally connected}
Let $X$ be a continuum. Then, $X$ is locally connected if and only if $\T$ is onto and idempotent on closed sets.
\end{theorem}

\begin{proof}
If $X$ is locally connected then $\T(K)=K$ for all $K\in 2^X$ \cite[Corollary 3.1.25]{Macias}. Thus, $\T$ is onto and idempotent on closed sets. Suppose $\T$ is onto and idempotent. By Theorem \ref{continuous}, $\T$ is continuous. Since $\T$ is onto, then necessarily $\T(\{x\})=\{x\}$ for all $x\in X$. Therefore by Theorem \ref{teorema2}, $X$ is locally connected  as the decomposition $\mathcal{D}$ is trivial.
\end{proof}

\begin{question}
Suppose $\T$ is onto, is $X$ locally connected? or equivalently, is $\T$ idempotent on closed sets?
\end{question}

The following result gives a positive answer to Problem 162 of the Houston Problem Book \cite[p. 390]{houston} and  characterizes all continua $X$ for which $\T$ is continuous.

\begin{theorem}\label{teorema3}
Let $X$ be a continuum. Then, $\T$ is continuous if and only if there exist a locally connected continuum $Y$ and a monotone and open map $f\colon X\to Y$ such that $\T(A)=f^{-1}(f(A))$, for each $A\in 2^{X}$.
\end{theorem}

\begin{proof}
Suppose that $\T$ is continuous. By Theorem \ref{teorema2}, $Y=X/\mathcal{D}$ is locally connected and $q\colon X\to Y$, the quotient map, is monotone and open. It is clear that $\T(\{x\})=q^{-1}(q(x))$, for each $x\in X$. Note that $X$ is $\T-$additive, by Corollary \ref{corolario1}. By \cite[Corollary 3.1.46]{Macias}, $\T(A)=\bigcup_{a\in A}\T(\{a\})$, for each $A\in 2^{X}$. Therefore, $\T(A)=\bigcup_{a\in A}\T(\{a\})=\bigcup_{a\in A}q^{-1}(q(a))=q^{-1}(q(A))$.

Conversely, suppose $f\colon X\to Y$ is a map as in the hypothesis.  Since $\T(A)=f^{-1}(f(A))$, for each $A\in 2^{X}$, and $f$ is open and surjective, then $\T$ is continuous by Theorems 1.8.22  and 1.8.24 of \cite{Macias}.
\end{proof}

With the next two results, we answer Questions 7.2.3 and 7.25 of \cite{Macias}, in positive.

\begin{theorem}\label{teorema descomponible}
Let $X$ be a continuum such that $\T\colon 2^X\to 2^X$ is continuous. If $\T(\{p\})$ has nonempty interior for some $p\in X$, then $X$ is indecomposable.
\end{theorem}

\begin{proof}
Since $\T$ is continuous, $\mathcal{D}=\{\T(\{x\}):x\in X\}$ is a decomposition of $X$, by Theorem \ref{teorema1}. Notice that if $\T(\{p\})=X$, then $\T(\{x\})=X$, for each $x\in X$, and $X$ is indecomposable \cite[Theorem 3.1.34]{Macias}. Thus, suppose that $\T(\{p\})\neq X$. Since $\T(\{p\})$ is closed, there exists $(x_n)_{n=1}^{\infty}\subseteq X\setminus \T(\{p\})$ such that $x_n\rightarrow x_0$, $x_0\in \T(\{p\})$. Since $\mathcal{D}$ is a decomposition (Theorem \ref{teorema1}), $\T(\{x_0\})=\T(\{p\})$ and $\T(\{x_n\})\cap \T(\{x_0\})=\emptyset$, for each $n\in\N$. Let $U$ be a nonempty open subset of $X$ such that $U\subseteq \T(\{x_0\})$. It is clear that $\T(\{x_0\})\in\langle X,U\rangle$ and, since $\T(\{x_n\})\cap \T(\{x_0\})=\emptyset$, $\T(\{x_n\})\cap U=\emptyset$, for each $n\in\N$. Thus, $\T(\{x_n\})\notin\langle X,U\rangle$ contrary to our assumption that $\T$ is continuous and $x_n\rightarrow x_0$. Therefore, $\T(\{p\})=X$ and $X$ is indecomposable.
\end{proof}

\begin{corollary}
\label{problema7.25}
Let $X$ be a continuum such that $\T\colon 2^X\to 2^X$ is continuous. Then, $X$ is decomposable if and only if $\T(\{p\})$ has empty interior, for each $p\in X$.
\end{corollary}

\begin{proof}
If $X$ is decomposable, then $\T(\{p\})$ has empty interior, for each $p\in X$, by Theorem \ref{teorema descomponible}. Conversely, suppose that $\T(\{p\})$ has empty interior, for each $p\in X$. Thus, $X/\mathcal{D}$ is a nondegenerate locally connected, where $\mathcal{D}=\{\T(\{x\}):x\in X\}$, by Theorem \ref{teorema2}. Therefore, $X/\mathcal{D}$ is decomposable \cite[Theorem 2, p. 207]{KW2} and, since the quotient map $q\colon X\to X/\mathcal{D}$ is monotone, we have that $X$ is decomposable.
\end{proof}

A continuum $X$ is \textit{irreducible} if there are two points $p$ and $q$ of $X$ such that no proper subcontinuum of $X$ contains both $p$ and $q$. A continuum $X$ is of \textit{type $\lambda$} provided that $X$ is irreducible and each indecomposable subcontinuum of $X$ has empty interior. It is known that a continuum $X$ is of type $\lambda$ if and only if admits a finest monotone upper semicontinuous decomposition $\mathcal{G}$ such that each element of $\mathcal{G}$ is nowhere dense and $X/\mathcal{G}$ is an arc \cite[Theorem 10, p. 15]{Thomas}. We say that a continuum $X$ is \textit{continuously irreducible} provided that $X$ is of type $\lambda$ and its decomposition $\mathcal{G}$ is continuous. We characterize continuously irreducible continua,  compare the following result with \cite[Theorem 3.2]{macias1}.

\begin{theorem}\label{irreducible}
Let $X$ be an irreducible continuum. Then, $X$ is either indecomposable or continuously irreducible continuum if and only if the set function $\T$ is continuous.
\end{theorem}

\begin{proof}
If $X$ is continuously irreducible, then $\T$ is continuous, by \cite[Theorem 3.2]{macias1}. It is clear that $\T$ is continuous, if $X$ is indecomposable \cite[Theorem 3.1.34]{Macias}. 

Conversely, suppose that $\T\colon 2^{X} \to 2^{X}$ is continuous and $X$ is not indecomposable. By Theorem \ref{teorema3}, there exist a locally connected continuum $Y$ and an open and monotone map $f\colon X\to Y$ such that $\T(A)=f^{-1}(f(A))$ for each $A\in 2^{X}$. Since $X$ is irreducible, $Y$ is irreducible, by \cite[Theorem 3, p.192]{KW2}. Hence, $Y$ is homeomorphic to $[0,1]$. Furthermore, $\T(\{x\})$ has empty interior for each $x\in X$, by Theorem \ref{teorema descomponible}. Thus, each indecomposable subcontinuum of $X$ has empty interior. Therefore, $X$ is a continuum of type $\lambda$ and $X$ is continuously irreducible, by \cite[Theorem 3.2]{macias1}.
\end{proof}

\noindent (J. Camargo)\par
\noindent Escuela de Matem\'aticas, Facultad de Ciencias, Universidad Industrial de
Santander, Ciudad Universitaria, Carrera 27 Calle 9, Bucaramanga,
Santander, A.A. 678, COLOMBIA.\par
\noindent e-mail: jcamargo@saber.uis.edu.co\par
\bigskip

\noindent (C. Uzc\'ategui) \par
\noindent Escuela de Matem\'aticas, Facultad de Ciencias, Universidad Industrial de
Santander, Ciudad Universitaria, Carrera 27 Calle 9, Bucaramanga,
Santander, A.A. 678, COLOMBIA.\par
\noindent e-mail: cuzcatea@saber.uis.edu.co\par
\bigskip

\end{document}